\documentclass[12pt,reqno]{amsart}
\usepackage[T1]{fontenc}
\usepackage{lmodern}
\usepackage{microtype}
\usepackage{amsmath,amssymb,mathtools}
\usepackage{enumitem}
\usepackage{hyperref}
\usepackage[margin=1in]{geometry}

\hypersetup{
  colorlinks=true,
  linkcolor=blue,
  citecolor=blue,
  urlcolor=blue
}

\newtheorem{theorem}{Theorem}[section]
\newtheorem{lemma}[theorem]{Lemma}
\newtheorem{proposition}[theorem]{Proposition}
\newtheorem{corollary}[theorem]{Corollary}
\newtheorem{conjecture}[theorem]{Conjecture}
\theoremstyle{remark}
\newtheorem{remark}[theorem]{Remark}

\newcommand{\dout}{d^{+}}
\newcommand{\din}{d^{-}}

\title[A $\sqrt{2}$-Approximation
to the Bilu-Linial Conjecture]
{A $\sqrt{2}$-Approximation to the Bilu-Linial Conjecture}
\author{Zejun Huang}
\date{November 2026
  }

\begin{document}

\begin{abstract}
   Bilu and Linial conjectured that
 if  $d\ge 2$, then every $d$-regular graph $G$ has an edge signing
$\sigma:E(G)\to\{-1,1\}$ such that its signed adjacency matrix $A_\sigma$ satisfies the Ramanujan bound
\begin{equation*}
  \rho(A_\sigma)\le 2\sqrt{d-1}
\end{equation*} and they proved that
\begin{equation*}
  \rho(A_\sigma)=O\!\left(\sqrt{d\log^3 d}\right).
\end{equation*}
 Using a method of interlacing polynomials,  Marcus, Spielman, and Srivastava confirmed one side of this conjecture that there is a signing $\sigma$ for which
 \[
    \lambda_{\max}(A_\sigma)\le 2\sqrt{d-1}.
\]
By constructing an auxiliary bipartite graph from a  balanced orientation of $G$ and applying  a bipartite signing argument,
we prove that every finite simple graph $G$ of maximum degree $d\ge 3$
admits an edge signing $\sigma$ such that
\[
    \rho(A_\sigma)
    \le 4\sqrt{\left\lceil \frac d2\right\rceil-1}
    \le 2\sqrt{2(d-1)}.
\]
This gives  a  new proof of  Ravichandran and Srivastava's bound $2\sqrt{2(d-1)}$.

\end{abstract}

\maketitle

\section{Introduction and main results}

A  $d$-regular graph is called a {\it $\lambda$-expander} if all its eigenvalues except the trivial one  $d$ lie in $[-\lambda,\lambda]$.
Expander graphs are highly connected sparse graphs that  have important applications in coding theory and theoretical computer science;   see the surveys of
Hoory, Linial, and Wigderson~\cite{HLW} and Lubotzky~\cite{LubotzkySurvey}.

The Alon-Boppana bound states that in every infinite family of connected $d$-regular graphs whose orders tend to infinity, the second largest adjacency eigenvalue has a lower limit at least $2\sqrt{d-1}$; see \cite{HLW, Nilli}.
A $d$-regular graph is called \emph{Ramanujan} when every nontrivial
eigenvalue lies in
\[
  [-2\sqrt{d-1},\,2\sqrt{d-1}].
\]  The first explicit infinite Ramanujan graph  families
were constructed by Lubotzky, Phillips, and Sarnak~\cite{LPS} and
independently by Margulis~\cite{Margulis}.  Morgenstern~\cite{Morgenstern}
subsequently constructed $(q+1)$-regular Ramanujan graphs for every prime
power $q$.  On the probabilistic side, Friedman~\cite{Friedman} proved
Alon's conjecture that random regular graphs are asymptotically almost
Ramanujan.

Bilu and Linial  \cite{BL} introduced a flexible approach based on repeated
$2$-lifts.
Let $G=(V,E)$ be an $n$-vertex graph with adjacency matrix $A$.   A {\it signing} of the edges of $G$ is a map
$\sigma:E\to\{-1,1\},
$
and the corresponding {\it signed adjacency matrix} is the symmetric matrix $A_\sigma$ with
\[
  (A_\sigma)_{uv}=
  \begin{cases}
    \sigma(uv),&\text{ if }uv\in E;\\
    0,&\text{ if }uv\notin E.
  \end{cases}
\]Given a signing $\sigma$ of $G$, denote by $A^+$ and $A^-$  the adjacency matrices of the positive and negative edge subgraphs, which are both symmetric 0-1 matrices.  Thus
\[
  A=A^++A^- ,\qquad A_\sigma=A^+-A^-.
\]
The {\it $2$-lift}  $\widehat G_\sigma$ corresponding to $\sigma$ is the $2n$-vertex graph with adjacency matrix
\[
  \widehat A_\sigma=
  \begin{pmatrix}
    A^+&A^-\\
    A^-&A^+
  \end{pmatrix}.
\]
Since $\widehat A_\sigma$ is orthogonally similar to $A\oplus A_{\sigma}$ through the matrix $
  U=\frac1{\sqrt2}
  \begin{pmatrix}
    I&I\\
    I&-I
  \end{pmatrix}
$, the spectrum of $\widehat A_\sigma$ is the  union of the spectra of $A$ and $A_\sigma$ with multiplicities counted.
This observation motivated the following
conjecture.

\begin{conjecture}[Bilu-Linial~\cite{BL}]\label{conj:BL}
Suppose $d\ge2$ is an integer. Then every $d$-regular graph $G$ has a signing
$\sigma:E(G)\to\{-1,1\}$ such that
\begin{equation}\label{eq:BL}
  \rho(A_\sigma)\le 2\sqrt{d-1}
\end{equation}
where $\rho(A_\sigma)$ denotes the spectral radius of $A_\sigma$.
\end{conjecture}

If this assertion holds at every stage of a lift tower, then repeated good
$2$-lifts of a Ramanujan base graph remain Ramanujan.
  Bilu and Linial  proved that every graph of
maximum degree $d$ has a signing $\sigma$ such that
\begin{equation}\label{eq:BL-original}
  \rho(A_\sigma)=O\!\left(\sqrt{d\log^3 d}\right),
\end{equation}
which led to nearly Ramanujan $2$-lifts and initiated the systematic study of
two-sided spectral signings.

By introducing a method of interlacing polynomials, Marcus, Spielman, and Srivastava \cite{MSS} proved that for every $d$-regular graph, there is a signing $\sigma$ such that
\begin{equation}\label{eq:3}
  \lambda_{\max}(A_\sigma)\le 2\sqrt{d-1},
\end{equation}
where $\lambda_{\max}(A_\sigma)$ denotes the maximum eigenvalue of $A_\sigma$.
Recall that the spectrum of every signed adjacency matrix of a bipartite graph is symmetric about
zero.   In the bipartite case, one can use \eqref{eq:3} to obtain the full Ramanujan
bound and prove the Bilu-Linial conjecture immediately.  Moreover, by  using repeated
good lifts one can generate infinite families of bipartite Ramanujan graphs of
every degree. Using interlacing families of mixed determinantal polynomials, Ravichandran and Srivastava proved $\rho(A_\sigma)\le 2\sqrt{2(d-1)}$.

In this paper, we obtain the following $\sqrt{2}$-approximation of the  Bilu-Linial conjecture by adopting a new strategy.

\begin{theorem}\label{thm:main}
Let $G$ be a finite simple graph of maximum degree $d\ge3$.  Then there exists a signing $\sigma:E(G)\to\{-1,1\}$ such that
\begin{equation}\label{eq:main-bound}
  \rho(A_\sigma)
  \le 4\sqrt{\left\lceil\frac d2\right\rceil-1}.
\end{equation}
Consequently,
\begin{equation}\label{eq:sqrt2-bound}
  \rho(A_\sigma)\le 2\sqrt{2(d-1)}.
\end{equation}
\end{theorem}

 Our proof combines graph-theoretic constructions with matrix techniques.
Starting from a balanced orientation of $G$, we construct an auxiliary bipartite graph whose signed biadjacency matrix $B_\sigma$ yields the decomposition $A_\sigma=B_\sigma+B_\sigma^T$.
The signed adjacency matrix of this bipartite graph has the block form
$\begin{pmatrix}
0&B_\sigma\\
B_\sigma^T&0
\end{pmatrix}$,
whose spectral norm equals that of $B_\sigma$.
The bipartite form of the Marcus-Spielman-Srivastava signing theorem~\cite{MSS} therefore provides a signing for which the spectral norm of $B_\sigma$ is bounded in terms of the maximum degree of the auxiliary graph.
The desired bound then follows from the triangle inequality for the spectral norm.
\section{Proofs}

In what follows, graphs are finite, simple, and undirected unless an orientation is explicitly specified.  For a real symmetric matrix $M$, its spectral radius equals its operator norm
\[
  \rho(M)=\lVert M\rVert=\max_i |\lambda_i(M)|.
\]
In a graph $G$, we denote by $d_{G}(u)$ the degree of a vertex $u$. In a digraph $D$, we denote by $d_D^+(u)$ and $d_D^-(u)$ the outdegree and the indegree of $u$, respectively.

We will need the following result, which can be obtained directly by applying Theorems 3.2, 3.7, 4.4 and 5.2 of \cite{MSS}.

\begin{lemma}\label{lem2.1}
Let $G$ be a graph of maximum degree $\Delta\ge2$.  There exists a signing $\tau$ of $G$ such that
\[
  \lambda_{\max}(A_\tau)\le 2\sqrt{\Delta-1}.
\]
\end{lemma}

Since the spectrum of every signed adjacency matrix of a bipartite graph is symmetric about
zero, Lemma \ref{lem2.1} implies the following corollary.

\begin{corollary}\label{co2.2}
Let $G$ be a bipartite graph of maximum degree $\Delta\ge2$.  Then there exists a signing $\tau$ of $G$ such that
\[
  \rho(A_\tau)=\lVert A_\tau\rVert\le 2\sqrt{\Delta-1}.
\]
\end{corollary}

An {\it orientation} $\omega$ of $G$ is the digraph obtained by assigning a direction to every edge of $G$.  We denote by $\dout_\omega(v)$ and $\din_\omega(v)$   the outdegree and indegree of $v$ in $\omega$, respectively.

Applying a standard Euler-tour construction, we have the following crucial lemma, which    is also a weak form of classical balanced-orientation results; see \cite{NashWilliams}. For the sake of readability, we present a proof here.

\begin{lemma}\label{lem:balanced-orientation}
Every finite graph $G$ has an orientation $\omega$ such that
\begin{equation}\label{eq:bo}
  \bigl|\dout_\omega(v)-\din_\omega(v)\bigr|\le1
  \quad\text{for all }\quad v\in V(G).
\end{equation}
Consequently,
\[
  \max\{\dout_\omega(v),\din_\omega(v)\}
  \le \left\lceil\frac{d_G(v)}2\right\rceil.
\]
\end{lemma}

\begin{proof}
It suffices to prove the case when $G$ is connected.  If $G$ has no vertex of odd degree, then we can get the required orientation by orienting all edges of $G$ along an Euler circuit.

Suppose $G$ has vertices of odd degree.
We construct a new graph $G'$ by adding one new vertex $u$, and adding a new edge between $u$ and   every odd-degree vertex of $G$.  Then all vertices of $G'$ have even degrees, since the number of odd-degree vertices is even.   Now we orient all edges of $G'$ along an Euler circuit. Then every vertex has equal indegree and  outdegree. Deleting $u$ and its incident arcs in this orientation of $G'$, we get an orientation $\omega$ of $G$ satisfying \eqref{eq:bo}.
\end{proof}

Let $\omega$ be an orientation of $G$. We define   the \emph{orientation bipartition graph} $G_\omega$ as follows. The vertex set is $V(G_{\omega})=V_L\cup V_R$ with $V_L$ and $ V_R$ being two copies of $V(G)$ relabelled by
\[
  V_L=\{u_L:u\in V(G)\},
  \qquad
  V_R=\{u_R:u\in V(G)\}.
\]
The edge set is $$E(G_{\omega})=\{u_Lv_R: \text{ whenever } u\to v \text{ is an arc in } \omega\}.$$
 Thus
\[
  d_{G_\omega}(u_L)=\dout_\omega(u),
  \qquad
  d_{G_\omega}(u_R)=\din_\omega(u).
\]
Set
\[
  \Delta_\omega
  \equiv \max_{u\in V(G)}
  \max\{\dout_\omega(u),\din_\omega(u)\}.
\]
Then $G_\omega$ is bipartite with maximum degree $\Delta_\omega$.

Given an orientation $\omega$ of $G$ and a  signing  $\tau$ of $G_{\omega}$, we can define a signing of $G$ as $$\sigma(uv)=\begin{cases}
 \tau(u_Lv_R),& \text{if } uv\in E(G) \text{ and }  u\to v \text{ is an arc in } \omega; \\
 \tau(v_Lu_R),& \text{if }  uv\in E(G) \text{ and }  v\to u \text{ is an arc in } \omega.
 \end{cases}$$
 Then $\sigma$ can be seen   identical with $\tau$.  Indexing both $V_L$ and $V_R$ by $V(G)$, we define the {\it signed biadjacency matrix} $B_\sigma$ by
\begin{equation}\label{eq:B-def}
  (B_\sigma)_{uv}
  =
  \begin{cases}
    \sigma(uv),&\text{if }u\to v\text{ in }\omega;\\
    0,&\text{otherwise}.
  \end{cases}
\end{equation}
Then the signed adjacency matrix of $G$ is
\begin{equation}\label{eq:symmetrization}
  A_\sigma=B_\sigma+B_\sigma^T,
\end{equation}
and the signed adjacency matrix of $G_{\omega}$ is
\begin{equation}\label{eq:bip-block}
  H_\tau
  =
  \begin{pmatrix}
    0&B_\sigma\\
    B_\sigma^T&0
  \end{pmatrix}.
\end{equation}
Moreover, we have
\begin{equation}\label{eq:block-norm}
  \lVert H_\tau\rVert=\lVert B_\sigma\rVert.
\end{equation}

Given an orientation $\omega$ of $G$, we can bound $\rho(A_\sigma)$ by $\Delta_\omega$ as follows.
\begin{proposition}\label{prop:orientation-bound}
Suppose $G$ has an orientation $\omega$ with $\Delta_\omega\ge2$.  Then $G$ has a signing $\sigma$ such that
\[
  \rho(A_\sigma)\le4\sqrt{\Delta_\omega-1}.
\]
\end{proposition}

\begin{proof}
Applying Corollary~\ref{co2.2} to $G_{\omega}$,  there is a signing $\tau$ of $G_\omega$ such that
\[
  \rho(H_\tau)=\lVert H_\tau\rVert
  \le2\sqrt{\Delta_\omega-1}.
\]
Let $\sigma$ be the corresponding signing of $G$ defined above.
Then
 by~\eqref{eq:block-norm}, we have
\[
  \lVert B_\sigma\rVert=  \lVert H_\tau\rVert
  \le2\sqrt{\Delta_\omega-1}.
\]
Since $A_\sigma$ is symmetric, by~\eqref{eq:symmetrization} and the triangle inequality we have
\[
  \rho(A_\sigma)
  =\lVert A_\sigma\rVert
  =\lVert B_\sigma+B_\sigma^T\rVert
  \le\lVert B_\sigma\rVert+\lVert B_\sigma^T\rVert
  =2\lVert B_\sigma\rVert\le4\sqrt{\Delta_\omega-1}.
\]
\end{proof}

 Now we are ready to present the proof of our main result.

\begin{proof}[Proof of Theorem~\ref{thm:main}]
Applying Lemma~\ref{lem:balanced-orientation}, we can choose a balanced orientation $\omega$ satisfying \eqref{eq:bo}.   If $G$ has maximum degree $d\ge3$, then
\[
  \Delta_\omega
  =\left\lceil\frac d2\right\rceil\ge 2.
\]
By Proposition~\ref{prop:orientation-bound}, we have
\[
  \rho(A_\sigma)
  \le4\sqrt{\Delta_\omega-1}
  \le4\sqrt{\left\lceil\frac d2\right\rceil-1}.
\]
\end{proof}

For regular graphs, by Theorem~\ref{thm:main} we have the following direct comparison with the Bilu-Linial conjecture.

\begin{corollary}\label{cor:regular}
Every finite $d$-regular graph with $d\ge3$ has a signing $\sigma$ such that
\[
  \rho(A_\sigma)
  \le4\sqrt{\left\lceil\frac d2\right\rceil-1}
  \le\sqrt2\,\bigl(2\sqrt{d-1}\bigr).
\]
\end{corollary}

\vskip 1cm

\begin{remark}
A sharper bound $\rho(A_\sigma)<(3+\sqrt{5})\sqrt{d-1}/2$ was proved by Lin and Zhou \cite{LZ} recently. Xu \cite{XU} constructed a 3-regular counterexample to Bilu and Linial's original conjecture.
\end{remark}

Zejun Huang\\
School of Mathematical Sciences\\
Shenzhen University\\
Shenzhen 518060, China \\
zejunhuang@szu.edu.cn
\end{document}